\documentclass[a4paper]{article}

\usepackage[english]{babel}
\usepackage[utf8x]{inputenc}
\usepackage[T1]{fontenc}

\usepackage[a4paper,top=3cm,bottom=2cm,left=3cm,right=3cm,marginparwidth=1.75cm]{geometry}

\usepackage[fleqn]{amsmath}
\usepackage{verbatim,amsthm}
\usepackage{graphicx}
\usepackage[colorinlistoftodos]{todonotes}
\usepackage[colorlinks=true, allcolors=blue]{hyperref}
\usepackage{amssymb,xcolor}

\newcommand{\fdd}{\stackrel{\mathrm{fdd}}{\longrightarrow}}
\newcommand{\tendsas}{\stackrel{a.s}{\longrightarrow}}
\newcommand{\tends}[1]{\xrightarrow[#1]{}}
\newcommand{\tendsd}{\xrightarrow{\ d\ }}
\newcommand{\eqd}{\stackrel{d}{=}}

\def\d{\hbox{\rm d}}
 
\DeclareMathOperator{\Var}{D}
\DeclareMathOperator{\Cov}{Cov}
\DeclareMathOperator{\Prob}{P} 
\DeclareMathOperator{\Mean}{E}

\newcommand{\abs}[1]{\lvert #1\rvert}
\newcommand{\Abs}[1]{\bigl| #1\bigr|}

\newcommand{\Nd}{\mathbb{N}}
\newcommand{\Rd}{\mathbb{R}}
\newcommand{\Zd}{\mathbb{Z}}
\newcommand{\calB}{\mathcal{B}}
\newcommand{\calF}{\mathcal{F}}

\newcommand{\calN}{\mathcal{N}}
\newcommand{\e}{\mathrm{e}}

\newtheorem{thm}{Theorem}
\newtheorem{lem}[thm]{Lemma}
\newtheorem{prop}[thm]{Proposition}
\newtheorem{dfn}[thm]{Definition}
\newtheorem{rem}[thm]{Remark}
\newtheorem{cor}[thm]{Corollary}

\newcommand{\bydef}{\stackrel{\mathrm{def}}{=}}
\newcommand{\imply}{\Rightarrow}

\title{On a covariance structure of some subset of self-similar Gaussian processes}
\author{V. Skorniakov}

\begin{document}
\maketitle

\begin{abstract}
We introduce a class of self-similar Gaussian processes and provide sufficient and necessary conditions for a member of the class
to admit a unique small scale limit in \(D[0;\infty)\). The class includes several well known processes. An example of application to the problem of estimation is given.
\end{abstract}

\section{Introduction}

Let \(\gamma\in(0;1),\sigma\in(0;\infty),l:[0;\infty)\to\Rd\) be fixed. Assume that \(l\) is measurable and that \(l(0)=1\). In this paper we consider some limit property of a centered Gaussian process $(X_t)_{t\geq 0}$ with $X_0\equiv 0$ and a covariance function given by\footnote{$l$ must also satisfy additional constraint imposed by positive definiteness of $R$;}\footnote{since $X_0\equiv 0$, $R(s,t)\equiv 0$ for $s \wedge t=0$}

\begin{equation}\label{l:coveq}
	R(s,t)=\sigma^2(s\wedge t)^{2\gamma}l\left(\frac{\vert s-t\vert}{s\wedge t}\right),s\wedge t>0.
\end{equation}
Several particularly well known examples admitting such representation are the following:
\begin{itemize}
\item sub-fractional Brownian motion (sfBm) $(S_t^H)$ with
\begin{align}
	R(s,t) = s^{2H}+t^{2H}-\frac{1}{2}\left[(s+t)^{2H}-\vert s-t\vert^{2H}\right],\gamma=H\in(0;1),\sigma^2=2-2^{2H-1},\nonumber\\
    l(u)=(2-2^{2H-1})^{-1}\left(1+(1+u)^2-\frac{1}{2}\left((2+u)^{2H}+u^{2H}\right)\right);\label{l:lOfsfbm}
\end{align}
\item bi-fractional Brownian motion (bfBm) $(B_t^{H,K})$ with $H\in(0;1),K\in(0;1]$,
\begin{align}
	R(s,t) = 2^{-K}\left(\left(s^{2H}+t^{2H}\right)^{K}-\vert s-t\vert^{2HK}\right),\gamma=HK\in(0;1),\sigma^2=1,\nonumber\\
    l(u)=2^{-K}\left(\left(1+(1+u)^{2H}\right)^{K}-u^{2HK}\right).\label{l:lOfbfbm}
\end{align}
\item Riemann Liouville process (RL) $(RL_t^{H})$ with,
\begin{align}
	R(s,t) = \frac{\int_0^{s \wedge t}\left((t-v)(s-v)\right)^{H-1/2} \d v}{\Gamma^{2} \left(H+\frac{1}{2}\right)}, \gamma = H \in (0;1),\sigma^2=\frac{1}{2H\Gamma^{2} \left(H+\frac{1}{2}\right)},\nonumber\\
    l(u)=2H \int_0^1 \left((v+u)v\right)^{H-1/2} \d v.\label{l:lOfRL}
\end{align}
\end{itemize}
Popularity of the families of processes above\footnote{to gain some insight of its magnitude, we offer to track the number of citing articles of the following short list of references: \cite{Bojdecki04}, \cite{NualartLei09},
\cite{Russo06}, \cite{Tudor07subFbm}, \cite{Tudor07bfBm} \cite{houdre02}, \cite{Robinson99}; for all but one the mentioned numbers obtained from Scopus at the date of submission are indicated in the list of references given at the end of the article}  was the main source of inspiration of our study. Some other are explained below. 

First of all, note that $R$ given by \eqref{l:coveq} defines a self-similar Gaussian process. The only self-similar Gaussian process with stationary increments is the fractional Brownian motion $(B_{t}^{H,1})$ (further on denoted as $(B_t^H)$). Hence, the class under consideration corresponds to Gaussian processes with non-stationary increments and in certain cases covariance suitable for modeling of long range dependence. Therefore it is interesting from both practical and theoretical point of view. Secondly, the structure of $R$ is completely determined by the self-similarity parameter $\gamma$ and function $l$. It is clear thus, that different properties of the members of the class could be expressed in terms of the analytic properties of $l$ and the restrictions on the range of $\gamma$. Since $l$ depends on a single variable, such characterization appeals to be well suited for applications giving the other reason for investigations.

To describe the purpose of the current paper, recall a concept of a small scale limit introduced in \cite{dobrushin80}. We say that a process $X=(X_t)_{t\geq0}$ admits a small scale limit ({ssl}) at $t_0\in[0;\infty)$, whenever there exists a normalization $a_{t_0}:(0;\infty)\to(0;\infty), a_{t_0}(u)\to 0+0,u\to 0+0$, and a process $Y^{t_0}=(Y_\tau^{t_0})_{\tau\geq0}$, such that
\begin{equation}\label{l:sslim}
	\left(\frac{X_{t_0}-X_{t_0+\tau u}}{a_{t_0}(u)}\right)\fdd (Y^{t_0}_\tau)_{\tau\geq0},
\end{equation}
where {fdd} stands for a convergence of finite dimensional distributions. It is needless to say that an existence of such limit is a favorable property admitting both practical and theoretical applications. Therefore present paper is devoted to the problem of this type. To be more precise, we provide sufficient and necessary conditions on $l$ ensuring that $X$ admits the small scale limit at each $t\geq0$. Moreover, it turns out that self-similarity, which is present in our case, enables to replace {fdd} convergence above by the stronger one, namely, weak convergence in the Skorohod space $D[0;\infty)$ (for details on this type of convergence consult subsection \ref{ss:weakConv}).

The paper is organized as follows. Section \ref{s:results} contains statement of the main result along with several examples of applications implied by an existence of {ssl}. Section \ref{s:facts} is a collection of auxiliary statements and definitions needed for the proofs. The latter are given in section \ref{s:proofs}.

\section{Results}\label{s:results}
Our main result is contained in the first two theorems given below. Before proceeding to the statement we provide several comments regarding the notions.

\begin{itemize}
\item Whenever it is possible and no confusion occurs, we omit time argument for the process and denote it by a single letter, e.g. $X$ is used instead of $(X_t)_{t\ge 0}$. The time argument always appears as a lower subscript with an upper ones left for the parameters upon which the process depends.
\item In all the rest part of the paper $\tendsd$ denotes weak convergence in $D[0;\infty)$ (see subsection \ref{ss:weakConv} for details) when used with a process type arguments. In case of random variables it denotes a common weak convergence. $\calF_D$ denotes the set of random elements of $D[0;\infty)$.
\item Though indirectly, it was already mentioned that the fBm $B^H,H\in(0;1)$, is obtained from $B^{H,K}$ by taking $K=1$. Consequently,
\begin{equation*}
	R_{B^H}(s,t)=\frac{1}{2}\left(s^{2H}+t^{2H}-\vert t-s\vert^{2H}\right).
\end{equation*}
It is convenient to extend this notion and allow $H$ attain value 1. In such case $B^1$ is defined by
\begin{equation*}
	B^1_t=tZ, \ Z\sim\mathcal{N}(0;1),t\geq 0.
\end{equation*}
It is obvious that $\lim\limits_{H\to 1-0}R_{B^H}=R_{B^1}$. The latter relationship justifies introduced extension.
\item Let $f:[0;A]\to\Rd$ for some $A \in(0;\infty)$. Then
\begin{gather*}
	\Delta f_{t,u}=f(t+u)-f(t),\Delta^{(2)}f_{t,u}=\Delta f_{t+u,u}-\Delta f_{t,u}=f(t+2u)-2f(t+u)+f(t)
\end{gather*}
provided $t,u \geq 0$ are such that $t+2u \in[0;A]$.
\item for any real valued $f,g$ notion $f \sim g,u \to u_0 $, means that $f(u)=g(u)(1+o(1)), u \to u_0$; the same applies to one sided limits.
\end{itemize}

\begin{thm}\label{t:main1}
Let $(X_t)_{t \geq 0}$ be Gaussian with a covariance defined by \eqref{l:coveq} and let $l(u)=\frac{1}{2}\left(1+(1+u)^{2 \gamma}-\left(u^\kappa L(u)\right)^2 \right)$ with some fixed $\kappa \in(0;1]$ and $L:(0;\infty)\to (0; \infty)$ slowly varying at zero. Fix $u,t \in(0;\infty)$ and define a random process $(Z^{t,u}_\tau)_{\tau \geq 0}$ by
\begin{equation*}
	Z_\tau^{t,u}=\frac{X_t-X_{t+u \tau}}{u^\kappa L(u)},\tau \geq 0.
\end{equation*}
Then $Z^{t,u}\tendsd c_t B^{\kappa}, u \to 0+0$, where $c_t=\sigma t^{\gamma-\kappa}$.
\end{thm}

\begin{thm}\label{t:main2}
Let $(X_t)_{t \geq 0}$ be Gaussian with a covariance defined by \eqref{l:coveq}. Assume that for all $t \in(0;\infty)$ there exist random process\footnote{$A>0$ is assumed to be fixed; its value is irrelevant since it suffices to have $a_t$ defined in some neighborhood of 0} $(Y_\tau^t)_{\tau \geq 0}$ and $a_{t}:(0;A)\to (0;\infty)$ such that:
\begin{itemize}
\item[(y1)] $Y_1^t$ is non-degenerate;
\item[(y2)] $Y^t \in \calF_D$;
\item[(a1)] $a_t(u)\tends{u \to 0+0}0+0$;
\item[(a2)] $\left(\frac{X_t-X_{t+ \tau u}}{a_t(u)}\right)_{\tau \geq 0}\tendsd Y^t, u \to 0+0$.
\end{itemize}
Then there exist $\kappa \in(0;1]$ and $L:(0;\infty)\to [0;\infty)$ such that
\begin{itemize}
\item[(i)] $L$ is slowly varying at zero;
\item[(ii)] $a_t(u)\sim c_t u^{\kappa}L(u),u \to 0+0$, with $c_t=\sigma t^{\gamma-\kappa} \sqrt{\frac{\Var{ Y_1^1}}{\Var{ Y_t^1}}}$;
\item[(iii)] $l(u)=\frac{1}{2}\left(1+(1+u)^{2 \gamma} - (u^{\kappa}L(u))^{2} \right),u>0$;
\item[(iv)] $\forall t \ Y^t$ is a constant multiple of $B^\kappa$ and it is a unique\footnote{up to a constant multiplier} small scale limit of $X$ at $t$.
\end{itemize}
\end{thm}

The rest results are devoted to demonstrate the use of {ssl} property and we treat them as examples.

\smallskip \emph{Example 1.} The framework is based on statistical applications met in practice and should be understood as follows. Let $T>0$ be fixed. Assume we have observations of $X$ at time points $T+\frac{k}{n}T,k=0,\dots,n$. The task is to estimate $\kappa$. Then one could make use of theorem \ref{t:clt} and corollary \ref{c:estimator}.

\begin{thm}\label{t:clt}
Assume that conditions of theorem \ref{t:main1} hold. Moreover, let $L$ satisfies the following additional constraints:
\begin{itemize}
\item[(L1)] $L(0)\bydef \lim_{u\to 0+0} L(u)$ exists, is positive and finite;
\item[(L2)] $L(u)=L(0) + o(\sqrt{u}),u\to 0+0$;
\item[(L3)] $\forall k\in\{2,\dots,n-2\},n\geq3$, $\forall u\in\left[0;\frac{1}{n}\right]$,
\begin{equation*}
	\frac{\Abs{\Delta^{(2)}p_{ku,u}-2(1+u)^{2\gamma}\Delta^{(2)}p_{(k-1)\frac{u}{1+u},\frac{u}{1+u}}+(1+2u)^{2\gamma}\Delta^{(2)}p_{(k-2)\frac{u}{1+2u},\frac{u}{1+2u}}}}{u^{2\kappa}}\leq c k^{-\zeta},
\end{equation*}
where $p(u)=\left(u^{\kappa}L(u)\right)^2$ whereas $c\geq0$ and $\zeta>\frac{1}{2}$ are fixed constants, independent of $k$ and $n$. 

Then
\begin{itemize}
\item[(i)] $R_n^{T}\bydef \frac{1}{n-2}\sum_{k=0}^{n-3}\psi\left(\Delta^{(2)}X_{T+\frac{k}{n}T,\frac{1}{n}T}, 
\Delta^{(2)}X_{T+\frac{k+1}{n}T,\frac{1}{n}T}\right)\tendsas \Lambda(\kappa)=\lambda(\rho(\kappa))$, where 
\begin{gather}\label{e:Lambda}
	\psi(x,y)=\frac{\abs{x+y}}{\abs{x}+\abs{y}},\\
	\lambda(r)=\frac{1}{\pi}\left(\arccos(-r)+\sqrt[]{\frac{1+r}{1-r}}\ln \left(\frac{2}{1+r}\right)\right),\\
    r(x)=\mathrm{corr}(\Delta^{(2)}B^{x}_{0,1},\Delta^{(2)}B^{x}_{1,1})=\frac{-7-9^{x}+4^{x+1}}{2(4-4^{x})},x\in(0;1);
\end{gather}
\item[(ii)] $\sqrt[]{n}\left(R_n^T-\Lambda(\kappa)\right)\tendsd \mathcal{N}(0;\Sigma(\kappa))$, where
\begin{equation}\label{e:Sigma}
	\Sigma(x)=\sum_{k\in\Zd}\Cov\left(\psi(\Delta^{(2)}B^{x}_{0,1},\Delta^{(2)}B^{x}_{1,1}),
    \psi(\Delta^{(2)}B^{x}_{k,1},\Delta^{(2)}B^{x}_{k+1,1})\right)
\end{equation}
\end{itemize}
\end{itemize}
\end{thm}

\begin{cor}\label{c:estimator}
$\widehat{\kappa}_n\bydef \Lambda^{-1}(R_n^T)\tendsas \kappa$ and
$\sqrt[]{n}(\widehat{\kappa}_n-\kappa)\tendsd\mathcal{N}(0;\Sigma(\kappa)(\Lambda^\prime(\Lambda^{-1}(\kappa)))^2)$.
\end{cor}

We provide several remarks.
\begin{itemize}
\item It is common to assume that one observes a trajectory of the process within\footnote{it is more convenient for us to denote an interval of observation $[0;2T]$ rather than $[0;T]$} $[0;2T]$. Theorem \ref{t:clt} therefore states that a researcher should collect data only within a second half of an interval of observation. The requirement may seem pretty strange and one could treat it as an artificial condition imposed by an imperfection of the method used to prove CLT. On the other hand, note that, with the fBm being an exception, the process under consideration is the one with non-stationary increments. Consequently, its behavior at the start of evolution is expected to be unpleasant and only after some time more stable one appears. Moreover, even discarding the first portion of data from $[0;T]$ (if such does exist) and applying theorem only to data from $[T;2T]$, one still retains the usual rate of convergence in CLT. Thus, it is very likely that for particular models from the introduced class the improvements of shrinkage of asymptotic confidence interval are possible only up to a constant multiplier with the order of shrinkage remaining $n^{-\frac{1}{2}}$. Practical superiority of estimating statistics based on data from $[0;2T]$ rather than $[T;2T]$ is also questionable because of the reasons mentioned above. That is, convergence to asymptotic distribution may be slower and/or more unstable giving a real gain only for very large data sets. In order to address these questions, simulation study is needed. However, this is not a topic of the present paper.
\item Theorem \ref{t:clt} is based on results of \cite{Surgailis11}. The latter were generalized in \cite{Surgailis13}. By making use of these one can deduce the whole class of statistics suitable for estimation of $\kappa$. An interested reader can find the corresponding example for the case of the fBm in \cite{KS017}. Although the method used there does not rely on results of \cite{Surgailis13} and takes a more quick direct approach by making use of results of \cite{arcones1994}, it is not difficult to see that application of results of \cite{Surgailis13} is an equivalent alternative.
\end{itemize}

\emph{Example 2.} We have already said that the main source of inspiration of our study was the popularity of the families of processes listed in the introduction. It is therefore not surprising to expect that these should possess the {ssl} property. An exact statement is given below.

\begin{prop}\label{p:sfBm_and_bfBm}
$\forall H \in(0;1), \forall K \in(0;1)$ covariances of processes $(S_t^H),(B_t^{H,K})$ and $(LR^H)$ admit representation with $l$ as in theorem \ref{t:main1}. The defining quantities are as follows:
\begin{itemize}
\item $L^2_{S^H}(u)=\frac{1}{2-2^{2H-1}}\left( 1+\left(\frac{2}{u}\right)^{2H}\left(\left(1+\frac{u}{2} \right)^{2H}-\frac{1+(1+u)^{2H}}{2}\right)\right), \kappa = H$;
\item $L^2_{B^{H,K}}(u)=2^{1-K}
\left[1+(\frac{1}{u})^{2HK}\left(
2^{K-1}(1+(1+u)^{2HK})-(1+(1+u)^{2H})^K
\right)
\right], \kappa = HK$;
\item $L^2_{LR^H}(u) = 2H\int_0^{1/u}\left[ v^{2H-1}+(v+1)^{2H-1}-2(v(1+v))^{H-1/2}\right]\d v $.
\end{itemize}
Moreover, corollary \ref{c:estimator} applies to all classes of processes as well provided $\kappa<\frac{3}{4}$.
\end{prop}

\section{Auxiliary facts}\label{s:facts}

Below we provide some results and concepts needed for the proofs of those stated in section \ref{s:results}. The references are given, however, we adopt the notions and statements to our context.

\subsection{Weak convergence in $D[0;\infty)$}\label{ss:weakConv}

Let $D[0; \infty)$ denotes a space of real-valued functions on $[0;\infty)$ which are continuous on the right and have limits on the left $\forall t \in (0;\infty)$. It is well known that $D[0; \infty)$ is metrizable by the Skorohod metric $d_S$ and that $(D[0;\infty), d_S)$ is separable. In our context it is natural to consider definition of $d_S$ given in \cite{Pollard84}, ch. VI.

\begin{dfn}
Let $T>0$ be finite and $x,y \in D[0;\infty)$. Define 
\begin{multline*}
	d_T(x,y)=\inf\{\delta>0\mid \text{there exist grids } 0=t_0<t_1<\dots<t_k,0=s_0<s_1<\dots<s_k:\\ 
    t_k,s_k \geq T, \abs{t_i-s_i}\leq \delta,\abs{x(t)-y(s)}\leq\delta, 
    t \in [t_i;t_{i+1}),s \in [s_i;s_{i+1})\}.
\end{multline*}
Then $d_S \bydef \sum_{k=1}^\infty 2^{-k}(d_k(x,y)\wedge 1)$. \qed
\end{dfn}

Let $\calB_S$ be the Borel $\sigma$-field induced by the topology of $(D[0;\infty),d_S)$, $(\Omega,\calF,\Prob)$ be a fixed probability space and $Z:\Omega \to D[0;\infty)$ be $(\calF,\calB_S)$ measurable. We say that $Z$ is a random process of $D[0; \infty)$ or, alternatively, a random element of $D[0;\infty)$. For any sequence $(Z_n)$ of random elements of $D[0;\infty)$ a weak convergence is denoted by $\tendsd$ and understood in a usual sense. That is\footnote{$\Rd$ is considered in usual way, i.e. as a metric space $(\Rd,d)$ with
$d(x,y)=\abs{x-y}$; then $\calB(\Rd)$ is a corresponding Borel $\sigma$-field},
\begin{equation*}
	Z_n \tendsd Z \Longleftrightarrow \Mean f(Z_n) \tends{n \to \infty} \Mean f(Z)\text{ for any bounded and continuous } f:D[0;\infty) \to \Rd.
\end{equation*}
The concept of weak convergence in $D[0; \infty)$ is stronger than convergence of finite dimensional distributions. The characterization is given by the theorem below (\cite{Pollard84}, ch. VI, lemma 9 and theorem 10; \cite{Billingsley95prob_and_measure}, theorem 8.2).

\begin{thm}\label{t:aux1_weak_conv}
Let $Z, Z_1,Z_2,\dots$ be random elements of $D[0; \infty)$ and $C_Z=\{0\}\cup \{t \in (0;\infty) \mid \Prob(Z_{t-}=Z_t)=1\}$ be the set of a.s. continuity of $Z$. Then $Z_n \tendsd Z$
iff the following conditions hold:
\begin{itemize}
\item[(i)] $Z_n \fdd Z$ on $C_Z$;
\item[(ii)] $\forall \epsilon,\delta>0,\forall a,b \in C_Z$ there exists a grid of points from $C_Z$ $a=t_0<t_1<\dots<t_k=b$ such that
\begin{gather*}
	\limsup_{n \to \infty}\Prob\left( \max_{1 \leq j \leq k} \Delta (Z_n,[t_{j-1};t_{j}])> \delta \right)<\epsilon\\
    where \\
    \Delta(x,[c;d]) \bydef \inf \{\eta >0\mid \exists s \in[c;d]: \sup_{c \leq t< s}\abs{x(t)-x(c)}< \eta, \sup_{s \leq t \leq d}\abs{x(t)-x(d)}< \eta\},
\end{gather*}
for $x \in D[0;\infty)$ and $0 \leq c<d < \infty$.
\end{itemize}
\end{thm}

\subsection{Tangent process}\label{ss:tangent_process}
Consider the setting of subsection \ref{ss:weakConv}. Let $\calF_D$ denotes the class of all random elements of $D[0; \infty)$ and $\calF_0=
\{Z \in \calF_D \mid Z_0 = 0 \ a.s.\}$

\begin{dfn}
 Fix $t \in[0;\infty)$. $Y \in \calF_0$ is called a tangent process of $Z \in \calF_D$ at $t$ provided there exist sequences $r_n \downarrow 0, q_n \downarrow 0, n \to \infty$, such that
 \begin{equation*}
	\left(\frac{Z_{t+ \tau r_n}-Z_t}{q_n} \right)_{\tau \geq 0} \tendsd Y.
\end{equation*}
The collection of all tangent processes at $t$ is denoted by $\mathrm{Tan}(Z,t)$. If there exists $Y \in \calF_0$ having property $\mathrm{Tan}(Z,t)=\{c Y \mid c \geq 0 \}$, one says that $Y$ is a unique tangent process of $Z$ at $t$ or, alternatively, that $\mathrm{Tan}(Z,t)$ is generated by $Y$. \qed
\end{dfn}
A thorough treatment of tangent processes can be found in \cite{Falconer02}--\cite{Falconer03}. To us the most important are the following results.

\begin{thm}[\cite{Falconer03}, proposition 3.3]\label{t:aux2_tangent1}
Let $t \in [0;\infty),Z \in \calF_D$ and $(Y_\tau)_{\tau \geq 0} \in \mathrm{Tan}(Z,t)$. Then $(cY_{s\tau})_{\tau \geq 0} \in \mathrm{Tan}(Z,t)$ for all $c \geq 0$ and all
$s>0$.
\end{thm}

\begin{thm}[\cite{Falconer03}, corollary 4.3]\label{t:aux3_tangent1}
Let $Z \in \calF_D$ be Gaussian and $\lambda$ denotes the Lebesgue measure on $[0;\infty)$. Then for $\lambda$  almost all $t$ at which $\mathrm{Tan}(Z,t)$ is generated by
the unique $Y^t$ one of the following holds:
\begin{itemize}
\item[(i)]  $Y^t_\tau \stackrel{a.s.}{=}\tau X$ with $X \sim \mathcal{N}(\mu;\sigma^2)$ for some $\mu \in \Rd$ and $ \sigma>0$;
\item[(ii)] there exists $H \in(0;1)$ such that $Y^t$ is a scalar multiple of $B^H$.
\end{itemize}
\end{thm}

\subsection{Excursion probabilities for continuous Gaussian processes}\label{ss:Gaussian_ex_prob}

In this subsection $(T,d)$ stands for a compact metric space whereas $Z=(Z_t)_{t \in T}$ denotes a continuous centered real-valued Gaussian process.

\begin{itemize}
\item The canonical pseudo metric on $T$ is defined by $\rho(s,t)=\sqrt{\Mean(Z_t-Z_s)^2}$.
\item $B_{\rho}(t,\epsilon) \bydef \{s \in T \mid \rho(s,t) \leq \epsilon\}$ is called the $\epsilon$-radius ball with respect to $\rho$ at $t$.
\item $N(t,\rho,\epsilon)=N(\epsilon)$ denotes the smallest number of balls having $\epsilon$ radius and covering $T$.
\end{itemize}
In what follows we make use of the theorem given below.

\begin{thm}[\cite{Adler07}, theorem 4.1.2]\label{t:aux4_ex_prob_bound}
Let $\sigma_T^2 \bydef \sup_{t \in T} \Mean Z_t^2, \bar{\Phi}(x)= \frac{1}{\sqrt{2 \pi}}\int_{-\infty}^x \e^{-y^2/2} \d y$. Suppose that for some $A> \sigma_T$, some $\alpha>0$ and some $\epsilon_0 \in(0; \sigma_T]$ it holds that $N(T, \rho, \epsilon) \leq (A/ \epsilon)^{\alpha}$ for all $ \epsilon \in [0; \epsilon_0)$. Then for $u \geq \frac{\sigma^2_T(1+ \sqrt{\alpha})}{\epsilon_0}$ it holds that
\begin{equation}\label{e:bound_on_ex_prob}
 \Prob \left( \sup_{t \in T} Z_t \geq u\right) \leq \left( \frac{KA u}{\sigma^2_T\sqrt{\alpha}}\right)^{\alpha}\bar{\Phi}\left( \frac{u}{\sigma_T}\right),
\end{equation}
where $K$ is a universal constant.
\end{thm}

\begin{rem}\label{r:on_Borel_TIS_inequality}
Due to continuity $\sup_{t \in T} Z_t$ is well defined and a.s. bounded. Moreover, it has finite expectation (see \cite{Adler07}, theorem 2.1.2). Since $Z$ is centered, 
\begin{equation*}
	\Prob\left( \sup_{t \in T} \abs{Z_t}>u \right) \leq 2\Prob\left( \sup_{t \in T} {Z_t}>u\right). 
\end{equation*}
Therefore \eqref{e:bound_on_ex_prob} also gives bound for $\Prob\left( \sup_{t \in T} \abs{Z_t}>u \right)$. \qed
\end{rem}

\subsection{Regular variation}

Recall that a measurable function $f:[A;\infty) \to (0;\infty), A \geq 0$, is called regularly varying at $\infty$ with index $\rho \in \Rd$ whenever $\frac{f(\lambda t)}{f(t)}\tends{t \to \infty} \lambda^\rho$ for any fixed $\lambda \in (0; \infty)$. If it is true that $\rho=0$, one says that $f$ is slowly varying at $\infty$. Each regularly varying function is of the form $f(t) = t^{\rho} g(t)$, where $g$ is slowly varying at $\infty$. 

A dual concept is that of regular (slow) variation at $0$. Namely, $f$ varies regularly at $0$ with index $\rho$ provided $t \mapsto f(t^{-1})$ varies regularly at $\infty$ with index $-\rho$. Because of this duality it suffices to formulate results for regular variation at $\infty$ as it is done below and then restate in an obvious way whenever it is required.

\begin{thm}[\cite{Goldie87}, theorem 1.2.1]\label{t:aux5_rv_on_compacts}
If $L$ is slowly varying then $\frac{L(\lambda t)}{L(t)}\tends{t \to \infty }1$ uniformly on each compact $\lambda$-set in $(0; \infty)$.
\end{thm}

\begin{thm}[\cite{Goldie87}, theorem 1.5.6]\label{t:aux6_rv_potter_bounds}
If $L$ is slowly varying then for any chosen constants $A>1,\delta>0$, there exists $c=c(A,\delta)$ such that 
\begin{equation*}
	\frac{L(x)}{L(y)} \leq A \max\left(\left(\frac{y}{x} \right)^{\delta},\left(\frac{y}{x} \right)^{-\delta}\right)\ \forall x,y \geq c.
\end{equation*}
\end{thm}

\subsection{Limit theorems}\label{ss:limit_theorems}

\begin{thm}[\cite{Surgailis11}, theorems 4.1--4.2]\label{t:aux7_IR_CLT}
Let $(Z_t)_{t \in [0;1]}$ be a centered real valued Gaussian process, $H: [0;1] \to (0;1),c:[0;1]\to (0;\infty)$. Assume\footnote{$[x]$ denotes an integer part of $x \in \Rd$;
in the proofs we also make use of $\{x\}=x-[x]$ --- the fractional part of $x$}
\begin{itemize}
\item[(A1)] for any fixed $k \in\{1,2,\dots\}$
\begin{align}
	 \lim_{n \to\infty} \sup_{t \in (0;1)} \sqrt{n}\left\vert \frac{ \Mean Z_{\frac{[nt]+j}{n}}-Z_\frac{[nt]}{n}}{(k/n)^{2H(t)}}-c(t)\right \vert =0; \label{e:bard_A11}\\
    \lim_{n \to\infty} \sup_{t \in (0;1)} \sqrt{n}\ln n\Abs{H(t)-H(t+n^{-1})}=0;\label{e:bard_A12}\\
    \lim_{n \to\infty} \sup_{t \in (0;1)} \sqrt{n}\ln n\Abs{c(t)-c(t+n^{-1})}=0;\label{e:bard_A13}
\end{align}
\item[(A2)] there exist $d>0$ and $\zeta>\frac{1}{2}$ such that $\forall j,k \in \{1,2,\dots,n\}: j \neq k, n \geq1$, it holds
\begin{equation}\label{e:bard_A2}
	\Abs{\Mean \Delta^{(2)} Z_{\frac{k}{n},\frac{1}{n}} \Delta^{(2)} Z_{\frac{j}{n},\frac{1}{n}}}\leq d \abs{j-k}^{-\zeta}\sqrt{\Var \Delta^{(2)} Z_{\frac{k}{n},\frac{1}{n}}\Var \Delta^{(2)} Z_{\frac{j}{n},\frac{1}{n}}}.
\end{equation}
Then
\begin{align*}
	R_n \bydef \frac{1}{n-2} \sum_{k=0}^{n-3}\psi\left(\Delta^{(2)} Z_{\frac{k}{n}},\Delta^{(2)} Z_{\frac{k+1}{n}}\right) \tendsas \int_{0}^{1} \Lambda(H(t)) \d t,\\
    \sqrt{n}\left(R_n-\int_{0}^{1} \Lambda(H(t)) \d t\right) \tendsd \calN\left(0;\int_{0}^{1} \Sigma(H(t)) \d t\right), n \to \infty,
\end{align*}
with $\psi(x,y),\Lambda(x)$ and $\Sigma(x)$ being the same as in the theorem \ref{t:clt}.
\end{itemize}
\end{thm}

\begin{thm}[\cite{Goldie87}, theorems 8.5.1--8.5.2]\label{t:aux8_ss_fdd}
Let $(Z_t)_{t>0}\in \calF_D$ and $f:(0;\infty) \to (0;\infty), g:(0; \infty) \to \Rd$ be measurable. Suppose that there exists $Y \in \calF_D$ such that:
\begin{itemize}
\item[(i)] $Y_1$ is non-degenerate;
\item[(ii)] $\left(\frac{Z_{ \tau u}-g(u)}{f(u)}\right)_{\tau>0} \fdd Y, u \to \infty$.
\end{itemize}
Then there exist constants $\rho,b,c \in \Rd$, such that $f$ is regularly varying with index $\rho$ and $\forall t>0$
\begin{equation}\label{e:aux8_ss_limit}
	(Y_{t \tau})_{\tau>0} \stackrel{\mathrm{fdd}}{=} \begin{cases}
		&(t^{\rho}(Y_\tau - c)+c)_{\tau>0}, \rho \neq0;\\
        &(Y_1 + b \ln \tau)_{\tau>0}, \rho = 0.
	\end{cases}
\end{equation}
In other words, $Y$ is self-similar with index $\rho.$
\end{thm}

\section{Proofs}\label{s:proofs}

\begin{proof}[Proof of theorem \ref{t:main1}]
By self similarity of $X$
\begin{equation*}
	\left(Z_{\tau}^{t,u}\right)\eqd t^{\gamma}\left(\frac{X_1-X_{1+\frac{u}{t}\tau}}{u^\kappa L(u)}\right)=
    t^{\gamma-\kappa}\left(\frac{X_1-X_{1+\frac{u}{t}\tau}}{\left(\frac{u}{t}\right)^\kappa L\left(\frac{u}{t}\right)}\right)\frac{L\left(\frac{u}{t}\right)}{L(u)}=
    t^{\gamma-\kappa}\frac{L\left(\frac{u}{t}\right)}{L(u)}(Z_{\tau}^{1,\frac{u}{t}}).
\end{equation*}
Therefore taking into account slow variation of $L$ it suffices to prove theorem for $t=1$. In what follows we accomplish this by checking that \emph{(i)--(ii)} of theorem \ref{t:aux1_weak_conv} hold. For short we omit time parameter and write $Z^{u}_\tau$ instead of $Z_\tau^{1,u}$.

\emph{(i).} Fix\footnote{to avoid inconsistencies put $0 \cdot L(0) \bydef 0$} $0<\tau_1\leq \tau_2<\infty$. Then $\Mean Z_{\tau_1}^{u}=\Mean Z_{\tau_2}^{u}=0$ and
\begin{multline*}
	\frac{u^{2 \kappa}L^2(u)}{\sigma^2}\Cov(Z_{\tau_1}^{u},Z_{\tau_2}^{u})=\\
    \frac{1}{\sigma^2}\Mean(X_1-X_{1+\tau_1 u})(X_1-X_{1+\tau_2 u})=\\
    1- l(\tau_1 u)-l(\tau_2 u)+(1+\tau_1 u)^{2 \gamma}l \left(\frac{( \tau_2- \tau_1)u}{1+ \tau_1 u}\right)=\\
    -\frac{1}{2}\Bigg[
    (1+\tau_1u)^{2 \gamma}+(1+\tau_2u)^{2 \gamma}-u^{2 \kappa } \left( \tau_1^{2 \kappa} L^2(\tau_1u)+\tau_2^{2 \kappa} L^2(\tau_2u)\right)-\\
    (1+ \tau_1 u)^{2 \gamma}\left( 1+\left(\frac{1+\tau_2u}{1+\tau_1u}\right)^{2 \gamma} - \left( \frac{(\tau_2-\tau_1)u}{1+\tau_1u}\right)^{2 \gamma} L^2\left( \frac{(\tau_2-\tau_1)u}	      	  {1+\tau_1 u}\right) \right)
    \Bigg]=\\
    \frac{u^{2 \kappa}}{2}\left[ \tau_1^{2 \kappa} L^2(\tau_1u)+\tau_2^{2 \kappa} L^2(\tau_2u) - (\tau_2- \tau_1)^{2 \kappa}(1+\tau_1 u)^{2(\gamma-\kappa)} L^2\left( \frac{(\tau_2-\tau_1)u}	      	  {1+\tau_1 u}\right)\right].
\end{multline*}
Since $L$ varies slowly at 0, theorem \ref{t:aux5_rv_on_compacts} implies
\begin{equation*}
	\frac{L(\tau_i u)}{L(u)} \tends{u \to 0+0}1,i=1,2, \quad \text{and} \quad \frac{L\left(\frac{(\tau_2-\tau_1)u}{1+\tau_1 u}\right)}{L(u)} \tends{u \to 0+0}1.
\end{equation*}
Thus,
\begin{equation*}
	\Cov(Z_{\tau_1}^u,Z_{\tau_2}^u)\tends{u \to 0+0} \frac{\sigma^2}{2}\left( \tau_1^{2 \kappa}+ \tau_2^{2 \kappa}-(\tau_2 - \tau_1)^{2 \kappa} \right)
\end{equation*}
or equivalently, $Z^u \fdd \sigma B^\kappa, u \to 0+0$.

\emph{(ii).} Fix $\epsilon,\delta>0, 0 \leq a < b <\infty$ and any $(u_n)_{n \geq1}:u_n \tends{n \to \infty} 0+0$. For the sake of clarity we split verification of \emph{(ii)} into several steps.

\emph{Step 1.} Let $0 \leq c< d<\infty$ be fixed. Then (because of a.s. continuity of $\tau \longmapsto Z^{u_n}_\tau$)
\begin{multline*}
 \left\{ 
 	\eta>0\ \Big\vert\ \exists s \in [c;d]: \sup_{c \leq \tau <s} \Abs{Z_{\tau}^{u_n}-Z_{c}^{u_n}}< \eta ,\sup_{s \leq \tau \leq d} \Abs{Z_{\tau}^{u_n}-Z_{d}^{u_n}} < \eta
 \right\} \supset \\
 \left\{ 
 	\eta>0 \ \Big\vert\  \sup_{c \leq \tau \leq d} \Abs{Z_{\tau}^{u_n}-Z_{c}^{u_n}}< \eta/2 
 \right\} \cup 
 \left\{
 	\eta>0 \ \Big\vert\  \sup_{c \leq \tau \leq d} \Abs{Z_{\tau}^{u_n}-Z_{d}^{u_n}} < \eta/2
 \right\}\ a.s.
\end{multline*}
Consequently,
\begin{multline*}
	\Prob \left( \Delta (Z^{u_n},[c;d])> \eta \right) \leq \Prob \left( \sup_{c \leq \tau \leq d} \Abs{Z_{\tau}^{u_n}-Z_{c}^{u_n}} \geq \frac{\eta}{2} \right) +
    \Prob \left( \sup_{c \leq \tau \leq d} \Abs{Z_{\tau}^{u_n}-Z_{d}^{u_n}} \geq \frac{\eta}{2} \right)\leq \\
    2 \left[ \Prob \left( \sup_{c \leq \tau \leq d} (Z_{\tau}^{u_n}-Z_{c}^{u_n}) \geq \frac{\eta}{2} \right) + \Prob \left( \sup_{c \leq \tau \leq d} (Z_{\tau}^{u_n}-Z_{d}^{u_n}) \geq \frac{\eta}{2} \right)\right],
\end{multline*}
with the last being true because of the fact that $Z^{u_n}$ is centered (see remark \ref{r:on_Borel_TIS_inequality}). Thus, in order to bound left hand side it suffices to bound each probability on the right hand side. We give a detailed implementation for the first one, since the other one is handled in the same way.

\emph{Step 2.} Let $0 \leq c < d<\infty$ be from the \emph{Step 1} and additionally satisfy 
\begin{equation}\label{e:constraint_on_cd}
	d-c < 4^{-\frac{1}{\kappa}}.
\end{equation}
Define a process $(\xi_t^{n,c,d})_{t \in [c;d]}$ by $\xi_{t}^{n,c,d} = Z_{t}^{u_n}-Z_c^{u_n},t \in[c;d]$.
Then $(\xi_{t}^{n,c,d})$ is centered, continuous and for $t \in [c;d]$, 
\begin{multline*}
	u_n^{2 \kappa}L^2(u_n)\Mean (\xi_{t}^{n,c,d})^2 = \Mean (X_{1+u_nt}-X_{1+u_nc})^2=\sigma^2(1+u_nc)^{2(\gamma - \kappa)}((t-c)u_n)^{2 \kappa}L^2 \left(\frac{(t-c)u_n}{1+u_n c}\right).
\end{multline*}
Since ${t-c} \in [0;1)$ and $u_n \to 0+0$, theorem \ref{t:aux6_rv_potter_bounds} implies that for some $n_0$ it holds
\begin{equation}\label{e:bound_by_potter}
	n \geq n_0 \longrightarrow\left(\frac{L \left(\frac{(t-c)u_n}{1+u_n c}\right)}{L(u)}\right)^2 \leq 2\sigma^{-2} (t-c)^{-\kappa} \vee (t-c)^{\kappa} \leq \frac{2\sigma^{-2}}{(t-c)^\kappa}.
\end{equation}
We can also assume that $n_0$ is chosen so that $(1+u_n c)^{2(\gamma - \kappa)}\leq 2$ for $n \geq n_0$. Then by \eqref{e:constraint_on_cd},
\begin{equation}\label{e:bound_on_s_cd}
	\sigma^2_{[c;d]} \bydef \sup_{t \in [c;d]}\Mean (\xi_{t}^{n,c,d})^2 \leq 4 (d-c)^{\kappa}<1.
\end{equation}
Next, note that for $c\leq s<t \leq d$ it holds $\xi_{t}^{n,c,d}-\xi_{s}^{n,c,d} = Z_t^{u_n}-Z_{s}^{u_n}=\xi_{t}^{n,s,d}$. Hence, the above yields that canonical metric\footnote{for the definition consult subsection \ref{ss:Gaussian_ex_prob}} of $\xi^{n,c,d}$ may be
bounded as follows:
\begin{equation*}
	\rho_{\xi^{n,c,d}}^2(s,t) \leq 4 \abs{t-s}^{\kappa}.
\end{equation*}
Consequently, the smallest number of balls having $\tilde{\epsilon} \in(0;1]$ radius with respect to $\rho_{\xi^{n,c,d}}$ and covering $[c;d]$ satisfies
\begin{equation*}
	N([c;d], \rho_{\xi^{n,c,d}}, \tilde{\epsilon}) \leq \left(\frac{A}{\tilde{\epsilon}}\right)^{\frac{2}{\kappa}},
\end{equation*}
with arbitrary chosen $A \geq 2(d-c)^{\frac{\kappa}{2}}$.
Let $\epsilon_0 = \sigma_{[c;d]}$. Then application of theorem \ref{t:aux4_ex_prob_bound} yields
\begin{equation*}
	\Prob \left( \sup_{t \in[c;d]} {\xi^{n,c,d}} \geq \eta \right) \leq \left(\frac{\kappa}{2}\frac{(K A \eta)^2}{\sigma^4_{[c;d]}}\right)^{1/\kappa} 
    \bar{\Phi}\left(\frac{\eta}{\sigma_{[c;d]}}\right)
\end{equation*}
for any $\eta \geq \sigma_{[c;d]}\left(1+\sqrt[]{\frac{2}{\kappa}}\right)$. In particular, setting $\eta=\theta_\kappa\sqrt{ \sigma_{[c;d]}}, \theta_{\kappa} = \left(1+\sqrt[]{\frac{2}{\kappa}}\right)$, and taking into account the bound \eqref{e:bound_on_s_cd}, one has
\begin{equation}\label{e:final_bound_on_ex_prob}
	\Prob \left( \sup_{t \in[c;d]} {\xi^{n,c,d}} \geq \theta_\kappa\sqrt[]{\sigma_{[c;d]}} \right) \leq \frac{\tilde{K}}{\sigma_{[c;d]}^{\frac{3}{\kappa}}}
    \bar{\Phi}\left(\frac{\theta_\kappa}{\sqrt{\sigma_{[c;d]}}}\right),
\end{equation}
for $n \geq n_0$ and $\tilde{K}=(\frac{\kappa}{2}(K A\theta_\kappa)^{2})^{\frac{1}{\kappa}}$. Finishing this step we note the following.
\begin{itemize}
\item The constant $\tilde{K}$ on the right hand side of \eqref{e:final_bound_on_ex_prob} may be regarded as a universal provided one neglects an obvious dependence on $\kappa$. For this is suffices to assume \eqref{e:constraint_on_cd}. Then taking $A\geq 1>\sigma_{[c;d]}$ we have that the constraint imposed on $A$  by the theorem \ref{t:aux4_ex_prob_bound} automatically holds.
\item $\sigma^2_{[c;d]}=O((d-c)^{\kappa})=o(1),(d-c) \to 0+0$.
\item $n_0$ assuring \eqref{e:final_bound_on_ex_prob} depends on $c$ and\footnote{because of \eqref{e:bound_by_potter} and condition $(1+u_n c)^{2(\gamma-\kappa)} \leq2, n \geq n_0$} $d-c$. It is an increasing function of both, however, in case of $d-c$ assumption \eqref{e:constraint_on_cd} discards the dependence on difference $d-c$.
\end{itemize}

\emph{Step 3.} Let integer $k\geq0 ,m \geq 1 $, be such that $[k;k+m] \supset [a;b]$ and $k$ is the biggest whereas $m$ is the smallest among all having this property. Partition each $[j;j+1],j=k,\dots,k+m-1$ into equal intervals $[t_l^j;t_{l+1}^j],l=0,\dots,q-1$, so that $4(t_{l+1}^j-t_l^j)^{\kappa}< 1 \wedge (\delta\theta_\kappa^{-1})^2$. Then $\forall j,l$
\begin{equation*}
	t_{l+1}^j-t_{l}^j<4^{-\frac{1}{\kappa}} \quad \text{and} \quad \theta_\kappa\sqrt[]{ \sigma_{[t_{l}^j;t_{l+1}^j]}}<\delta.
\end{equation*}
Therefore  application of the results obtained in the previous steps (varying constant value from line to line is denoted by the same letter $K$ as long as its magnitude does not affect the limit) yields 

\begin{multline*}
	\Prob \left( \max_{j,l} \Delta (Z^{u_n},[t_l^j;t_{l+1}^j]>\delta)\right) \leq K \cdot q \cdot m \max_{j,l}\sigma_{[t_l^j;t_{l+1}^j]}^{-\frac{3}{\kappa}}\bar{\Phi}\left(\frac{\theta_{\kappa}}{\sqrt[]{\sigma_{[t_l^j;t_{l+1}^j]}}}\right)\leq\\ 
    K (b-a)\max_{j,l}\sigma_{[t_l^j;t_{l+1}^j]}^{-\frac{5}{\kappa}}\bar{\Phi}\left(\frac{\theta_\kappa}{\sqrt[]{\sigma_{[t_l^j;t_{l+1}^j]}}}\right),
\end{multline*}
since $q^{-1}=t_{l+1}^j-t_l^j$ for all $j,l$ and $m$ is proportional to $(b-a)$. It is clear that $x^{\frac{5}{\kappa}}\bar{\Phi}(\theta_\kappa\sqrt[]{x})\tends{x \to \infty}0$. Hence, if there is a need, one can increase the value of $q$ up to the smallest integer for which the right hand side does not exceed $\epsilon$. Then it remains to pass to the upper limit as $n \to \infty$.
\end{proof}

\begin{proof}[Proof of theorem \ref{t:main2}]
\emph{Step 1.} Fix $t \in(0;\infty)$ and define a random process $(Z^t_\tau)_{\tau>0}$ by
\begin{equation*}
	Z^t_\tau = X_t-X_{t+\frac{1}{\tau}},\tau>0.
\end{equation*}
Let $y \to \infty$. Put $u = y^{-1},f_t(y) = a_t(u)$. Then
\begin{equation*}
	\left(\frac{Z^t_{\tau y}}{f_t(y)}\right)_{\tau>0} \fdd \left( Y_{\frac{1}{\tau}}^t\right)_{\tau>0}.
\end{equation*}
Therefore theorem \ref{t:aux8_ss_fdd} implies that $\left( Y_{\frac{1}{\tau}}^t\right)_{\tau>0}$ is self-similar with some index $\kappa_t\in \Rd$ and $f_t$ is regularly varying with $\kappa_t$. Consequently, $\left( Y_{\tau}^t\right)_{\tau>0}$ is self-similar with index $-\kappa_t$ and $a_t$ is regularly varying at 0 with $-\kappa_t$. Since 
\begin{equation*}
	Y^t_0 \eqd \lim_{u \to 0+0}\frac{X_t-X_{t+0 \cdot u}}{a_t(u)}\equiv 0,
\end{equation*}
$\left( Y_{\tau}^t\right)_{\tau \geq 0}$ is also self-similar with index $-\kappa_t$. Moreover, assumption $a_t(u) \tends{u \to 0+0}0+0$ implies $\kappa_t \leq 0$. In fact, one must necessary
have $\kappa_t<0$. Indeed, if it were true that $\kappa_t = 0$, then by theorem \ref{t:aux8_ss_fdd} it were true that $Y^t_\tau \eqd Y^t_1 + b \ln\tau,\tau>0$. Since the limit of Gaussian process is Gaussian, \emph{fdd} convergence yields convergence of the first two moments. Thus, for any $\tau>0$,
\begin{equation*}
	\Mean \left(\frac{X_t-X_{t+\tau u}}{a_t(u)}\right) \tends{u \to 0+0}\Mean Y^t_{\tau} = 0 \Rightarrow 0 = \Mean Y^t_\tau = \Mean Y_1^t + b \ln \tau = b \ln \tau \Rightarrow
    b = 0 \Rightarrow Y_\tau^t \eqd Y^t_{1}.
\end{equation*}
By assumption, $Y^t_{1}$ is non-degenerate. On the other hand, continuity of the paths on the right yields
\begin{equation*}
	\Prob \left(Y_1^t = 0 \right) = \Prob \left( \lim_{\tau \to 0+0}Y_\tau^t = 0 \right) = 1.
\end{equation*}
Obtained contradiction excludes the case $\kappa_t = 0$. Also note that $c$ in \eqref{e:aux8_ss_limit} is equal to 0 because of the same condition $Y^t_0 = 0$.

\emph{Step 2.} Fix $t \in(0;\infty)$. Then results of \emph{Step 1} yield
\begin{equation*}
	\frac{X_t-X_{t+tu}}{a_t(u)} \tendsd Y^t_t \Rightarrow \Var \left(\frac{X_t-X_{t+tu}}{a_t(u)}\right) \to \Var Y^t_t = t^{-2 \kappa_t} \Var Y^t_{1}.
\end{equation*}
On the other hand by self-similarity of $X$,
\begin{equation*}
	\Var \left(\frac{X_t-X_{t+tu}}{a_t(u)}\right) = t^{2 \gamma}\frac{a_1^2(u)}{a_t^2(u)}\Var \left(\frac{X_{1}-X_{1+u}}{a_1(u)}\right) \sim
    t^{2 \gamma}\frac{a_1^2(u)}{a_t^2(u)}\Var Y^1_{1}.
\end{equation*}
Thus,
\begin{equation*}
	a_t(u) \sim t^{\gamma + \kappa_t} a_1(u)\sqrt[]{\frac{\Var Y^1_1}{\Var Y^t_1}} =
    t^{\gamma + \kappa_t} u^{-\kappa_1}L_1(u)\sqrt[]{\frac{\Var Y^1_1}{\Var Y^t_1}},
\end{equation*}
where $L_1$ is slowly varying at 0 by the \emph{Step 1}. Hence, $\forall t \ \kappa_t=\kappa_1 \bydef  -\kappa $.

\emph{Step 3.} 
\begin{multline*}
	\Var(X_1-X_{1+u})=\sigma^2(1+(1+u)^{2 \gamma}-2 l(u)) \Rightarrow l(u) = \frac{1}{2}\left( 1+(1+u)^{2 \gamma} -\frac{\Var(X_1-X_{1+u})}{\sigma^2}\right).
\end{multline*}
By all above, $\Var(X_1-X_{1+u}) \sim a_1^2(u) \sim u^{2  \kappa}L_1^2(u)$. Therefore $L(u) \bydef \sqrt{\frac{\Var(X_1-X_{1+u})}{\sigma^2 u^{2 \kappa}}}$ varies slowly at 0.

\emph{Step 4}. It remains to prove the last claim. Fix $t \in(0;\infty)$. Note that $c_t Y^t \in \mathrm{Tan}(X,t)$ with any $r_n \downarrow 0$ and $q_n \stackrel{c_t \neq 0}{=} c_t^{-1} a_t(r_n)$. Take arbitrary $\tilde{Y} \in \mathrm{Tan}(X,t)$. If $r_n \downarrow 0, q_n \downarrow 0$, are such that $\left(\frac{X_{t+\tau r_n}-X_{t}}{q_n}\right)_{\tau \geq 0}\tendsd \tilde{Y}$
then Gaussianity yields $\Var\left(\frac{X_{t}-X_{t+t r_n}}{q_n}\right)\tends{n \to \infty} \Var\tilde{Y}_t$. If $\Var\tilde{Y}_t>0$, then 
\begin{multline*}
	t^{2 \gamma}\frac{a_1^2(r_n)}{q_n^2} \Var Y_1^1 \sim t^{2 \gamma}\frac{a_1^{2}(r_n)}{q_n^2}\Var \left(\frac{X_1-X_{1+r_n}}{a_1(r_n)}\right) = 
    \Var \left(\frac{X_t-X_{t+tr_n}}{q_n}\right) \to \Var \tilde{Y}_t>0 \Rightarrow \\
    q_n \sim t^{\gamma}\sqrt{\frac{\Var Y_1^1}{\Var \tilde{Y}_t}} a_1(r_n).
\end{multline*}
Consequently, $\tilde{Y}$ is a constant multiple of $Y^t$. If $\Var \tilde{Y}_t =0$, then by theorem \ref{t:aux2_tangent1} and self-similarity of $X$, $\Var \tilde{Y}_\tau=0$ for all $\tau>0$. Thus, $\tilde{Y}$ is zero multiple of $Y^t$. Summing up, $Y^t$ is a unique tangent process of $X$ at $t$. By theorem \ref{t:aux3_tangent1}, the set of such $t \in(0; \infty)$ for which $Y^t$ is not a scalar multiple of $B^{\kappa}$ has the Lebesgue measure 0. In our case self-similarity of $X$ implies that this set is empty. Indeed, fix arbitrary $t_0$ having property $Y^{t_0} \eqd c_{t_0}B^{\kappa}$ and take any $t_1 \in(0;\infty) \setminus \{t_0 \}$. Then by noting that
\begin{multline*}
	\left(\frac{X_{t_1}-X_{t_1+\tau u}}{a_{t_1}(u)}\right)_{\tau \geq 0} \eqd \left(\frac{t_1}{t_0}\right)^{\gamma}\left(\frac{X_{t_0}-X_{t_0+\tau \frac{u t_0}{t_1}}}{a_{t_1}(u)}\right)_{\tau \geq 0} =\\
    c_{t_0,t_1}(u)\left(\frac{X_{t_0}-X_{t_0+\tau \frac{u t_0}{t_1}}}{a_{t_0}(\frac{t_0 u}{t_1})}\right)_{\tau \geq 0} \tendsd \tilde{c}_{t_0,t_1}c_{t_0} B^{\kappa},u \to 0+0,
\end{multline*}
where $c_{t_0,t_1}(u)\tends{u \to 0+0} \tilde{c}_{t_0,t_1}\in(0;\infty)$, one obtains the claim.
\end{proof}

\begin{proof}[Proof of theorem \ref{t:clt}]
Define a process $(Z_t^T)_{t \in [0;1]}$ by $Z^T_t = X_{T+tT},t \in[0;1]$. Below we show that under assumptions made above, theorem \ref{t:aux7_IR_CLT} applies to $Z^T$ with $H(t) \equiv \kappa$ and $c(t) = (\sigma L (0)(1+t)^{\gamma-\kappa})^2$. Note that in some expressions time argument of $Z^T$ falls into the range of its domain only asymptotically. If this is the case, we do not comment keeping in mind that the mentioned expressions are well defined provided $n$ is large enough. For short we assume that $T=1$ and denote $Z^1$ by $Z$. The case of $T \neq 1$ reduces to this one because of self-similarity.

\smallskip\emph{(A1).} Fix $k \in \{1,2,\dots\}$ and $t \in(0;1)$. Then
\begin{multline}\label{e:var_of_first_diff}
	\Mean \left(Z_{\frac{[nt]+k}{n}}-Z_{\frac{[nt]}{n}}\right)^2=
    \Mean \left(Z_{t-\frac{ \{nt\}-k}{n}}-Z_{t-\frac{\{nt\}}{n}}\right)^2=\\
    \Mean \left(X_{1+t-\frac{ \{nt\}-k}{n}}-X_{1+t-\frac{\{nt\}}{n}}\right)^2=
    \left[ n\left(1+t-\frac{\{nt\}}{n}\right)=\frac{1}{u_n^t}\right]=\\
    (nu_n^t)^{-2 \gamma} \Mean(X_{1+ku_n^t}-X_1)^2=\sigma^2(ku_n^t)^{2 \kappa}(nu_n^t)^{-2 \gamma}L^2(ku_n^t)=\\
    \sigma^2\left(\frac{k}{n}\right)^{2 \kappa}(nu_n^t)^{2(\kappa-\gamma)}L^2(ku_n^t).
\end{multline}
Since $(nu_n^t)^{-1}=1+t-\frac{\{nt\}}{n}=1+t+O\left(\frac{1}{n}\right)$ and $k$ is fixed,
\begin{multline*}
	\frac{\Mean \left(Z_{\frac{[nt]+k}{n}}-Z_{\frac{[nt]}{n}}\right)^2}{\left(\frac{k}{n}\right)^{2 \kappa}}=\sigma^2\left(1+t+O\left(\frac{1}{n}\right) \right)^{2(\gamma - \kappa)}\left(L^2(0)+o(\sqrt[]{ku_n^t}) \right)=\\
    \sigma^2(1+t)^{2(\gamma-\kappa)}\left(1+O\left(\frac{1}{n}\right)\right)\left(L^2(0)+o\left(\frac{1}{\sqrt[]{n}}\right) \right)=\sigma^2L^2(0)(1+t)^{2(\gamma-\kappa)}+o\left(\frac{1}{\sqrt[]{n}}\right).
\end{multline*}
Thus, \eqref{e:bard_A11} holds. \eqref{e:bard_A12} is trivial. Finally \eqref{e:bard_A13} follows easily by noting that $c(t)=\sigma^2L^2(0)(1+t)^{2(\gamma-\kappa)},t \in[0;1]$, is continuously differentiable on $[0;1]$.

\emph{(A2).} Let $j \in \{0,1, \dots, n-2\},u_{n,j}=\frac{1}{n+j}$. Then $ u_{n,j+m}=\frac{u_{n,j}}{1+mu_{n,j}},m\geq0$. Thus, for $k \geq1$,
\begin{multline}\label{e:cov_of_first_diff}
	\sigma^{-2}\Mean \Delta Z_{\frac{j}{n},\frac{1}{n}}\Delta Z_{\frac{j+k}{n},\frac{1}{n}}=
    \sigma^{-2}\Mean \left(X_{1+\frac{j+1}{n}}-X_{1+\frac{j}{n}}\right)\left(X_{1+\frac{j+k+1}{n}}-X_{1+\frac{j+k}{n}}\right)=\\
    \left(1+\frac{j+1}{n}\right)^{2 \gamma}\left( l\left( \frac{\frac{k}{n}}{1+\frac{j+1}{n}}\right)-l\left( \frac{\frac{k-1}{n}}{1+\frac{j+1}{n}}\right)\right)-
    \left(1+\frac{j}{n}\right)^{2 \gamma}\left( l\left( \frac{\frac{k+1}{n}}{1+\frac{j}{n}}\right)-l\left( \frac{\frac{k}{n}}{1+\frac{j}{n}}\right)\right)=\\
    n^{-2 \gamma}\left( 
    	u_{n,j+1}^{-2\gamma}(l(k u_{n,j+1})-l((k-1)u_{n,j+1}))-u_{n,j}^{-2\gamma}(l((k+1)u_{n,j})-l(k u_{n,j}))
    \right)=\\
    \frac{(nu_{n,j})^{-2 \gamma}}{2}\Big[ 
    	(1+u_{n,j})^{2 \gamma}\left( \frac{(1+(k+1)u_{n,j})^{2 \gamma}-(1+k u_{n,j})^{2 \gamma}}{(1+u_{n,j})^{2 \gamma}} - \Delta p_{(k-1)\frac{u_{n,j}}{1+u_{n,j}},\frac{u_{n,j}}{1+u_{n,j}}}\right) - \\
        \left( (1+(k+1)u_{n,j})^{2 \gamma} - (1+ku_{n,j})^{2 \gamma} -  \Delta p_{ku_{n,j},u_{n,j}} \right)
        \Big]  = \\
        \frac{(nu_{n,j})^{-2 \gamma}}{2}\left(
        \Delta p_{ku_{n,j},u_{n,j}} - (1+u_{n,j})^{2 \gamma}\Delta p_{(k-1)\frac{u_{n,j}}{1+u_{n,j}},\frac{u_{n,j}}{1+u_{n,j}}}
    	\right).
\end{multline}
Let $u_n^t,t \in (0;1)$ be as in \emph{(A1)}. Then $u_{n}^{\frac{j}{n}}=u_{n,j}$ and taking $k=1$ in \eqref{e:var_of_first_diff} together with \eqref{e:cov_of_first_diff} yields
\begin{equation*}
\sigma^{-2}\Var \Delta Z_{\frac{j}{n},\frac{1}{n}} = u_{n,j}^{2 \kappa}(nu_{n,j})^{-2 \gamma}L^2(u_{n,j});
\end{equation*}	
\begin{multline*}
    \sigma^{-2}\Var \Delta^{(2)} Z_{\frac{j}{n},\frac{1}{n}} = \sigma^{-2}\left(\Var \Delta Z_{\frac{j+1}{n},\frac{1}{n}}+\Var \Delta Z_{\frac{j}{n},\frac{1}{n}}-2 \Mean \Delta Z_{\frac{j+1}{n},\frac{1}{n}}\Delta Z_{\frac{j}{n},\frac{1}{n}}\right)=\\
    \left(\frac{u_{n,j}}{1+u_{n,j}}\right)^{2 \kappa}\left(\frac{nu_{n,j}}{1+u_{n,j}}\right)^{-2 \gamma}L^{2}\left(\frac{u_{n,j}}{1+u_{n,j}}\right)-\\
    2\frac{(nu_{n,j})^{-2 \gamma}}{2}\left( \Delta p_{u_{n,j},u_{n,j}}-
    (1+u_{n,j})^{2 \gamma}\Delta p_{0,\frac{u_{n,j}}{1+u_{n,j}}}\right)+(u_{n,j})^{2 \kappa}(nu_{n,j})^{-2 \gamma}L^2(u_{n,j})=\\
    (u_{n,j})^{2 \kappa}(nu_{n,j})^{-2 \gamma}\left( 2(1+u_{n,j})^{2( \gamma-\kappa)}L^2\left(\frac{u_{n,j}}{1+u_{n,j}}\right)+2L^2(u_{n,j}) + 2^{2 \kappa}L^2(2 u_{n,j}) \right) \stackrel{n \to \infty}{\sim} \\
    L^2(0)(4+4^{\kappa})(u_{n,j})^{2 \kappa}(nu_{n,j})^{-2 \gamma}.
\end{multline*}
Therefore for $k \in \{2,\dots,n\}:k+j \leq n$,
\begin{multline*}
	\Mean \Delta^{(2)}Z_{\frac{j}{n},\frac{1}{n}}\Delta^{(2)}Z_{\frac{j+k}{n},\frac{1}{n}} {=} 
    \Mean \left( \Delta Z_{\frac{j+1}{n},\frac{1}{n}} - \Delta Z_{\frac{j}{n},\frac{1}{n}}\right) \left( \Delta Z_{\frac{j+k+1}{n},\frac{1}{n}} - \Delta Z_{\frac{j+k}{n},\frac{1}{n}}\right)=\\
    \frac{(nu_{n,j+1})^{-2 \gamma}}{2}\Bigg[
    \Delta p_{ku_{n,j+1},u_{n,j+1}}-(1+u_{n,j+1})^{2 \gamma}\Delta p_{(k-1)\frac{u_{n,j+1}}{1+u_{n,j+1}},\frac{u_{n,j+1}}{1+u_{n,j+1}}}-\\
    \left(\Delta p_{(k-1)u_{n,j+1},u_{n,j+1}}-(1+u_{n,j+1})^{2 \gamma}\Delta p_{(k-2)\frac{u_{n,j+1}}{1+u_{n,j+1}},\frac{u_{n,j+1}}{1+u_{n,j+1}}}\right)
    \Bigg]-\\
   \frac{(nu_{n,j})^{-2 \gamma}}{2}\Bigg[
    \Delta p_{(k+1)u_{n,j},u_{n,j}}-(1+u_{n,j})^{2 \gamma}\Delta p_{k\frac{u_{n,j}}{1+u_{n,j}},\frac{u_{n,j}}{1+u_{n,j}}}-\\
    \left(\Delta p_{ku_{n,j},u_{n,j}}-(1+u_{n,j})^{2 \gamma}\Delta p_{(k-1)\frac{u_{n,j}}{1+u_{n,j}},\frac{u_{n,j}}{1+u_{n,j}}}\right)
    \Bigg]=
-\frac{(nu_{n,j})^{-2 \gamma}}{2}\Big[
\Delta^{(2)} p_{ku_{n,j},u_{n,j}}-\\
2(1+u_{n,j})^{2 \gamma}\Delta^{(2)} p_{(k-1)\frac{u_{n,j}}{1+u_{n,j}},\frac{u_{n,j}}{1+u_{n,j}}}+
(1+2u_{n,j})^{2 \gamma}\Delta^{(2)} p_{(k-2)\frac{u_{n,j}}{1+2u_{n,j}},\frac{u_{n,j}}{1+2u_{n,j}}}
\Big] \stackrel{n \to \infty}{\sim}\\
K\cdot \cdot\sqrt[]{\Var \Delta^{(2)}Z_{\frac{j}{n},\frac{1}{n}} \Var \Delta^{(2)}Z_{\frac{j+k}{n},\frac{1}{n}}} (u_{n,j})^{-2 \kappa}\Big[
\Delta^{(2)} p_{ku_{n,j},u_{n,j}}-\\
2(1+u_{n,j})^{2 \gamma}\Delta^{(2)} p_{(k-1)\frac{u_{n,j}}{1+u_{n,j}},\frac{u_{n,j}}{1+u_{n,j}}}+
(1+2u_{n,j})^{2 \gamma}\Delta^{(2)} p_{(k-2)\frac{u_{n,j}}{1+2u_{n,j}},\frac{u_{n,j}}{1+2u_{n,j}}}
\Big],
\end{multline*}
where $K = K(\kappa,j,k,n)$ is uniformly bounded for all $j,k,n$.
\end{proof}

\begin{proof}[Proof of corollary \ref{c:estimator}.] 
To prove the corollary simply apply the Delta method.
\end{proof}

In order to prove proposition \ref{p:sfBm_and_bfBm} we need the following lemma. We believe that it is proved elsewhere in a form suitable for our needs, however, we couldn't find a corresponding reference. Therefore we provide a proof for completeness.
\begin{lem}\label{l:diff2_asymp}
Let $x>0$ be fixed and $g^x:[0;\infty) \to \Rd$ be defined by 
\begin{equation*}
	g^x(y) = (y+2)^x-2(y+1)^x + y^x.	
\end{equation*}
Then one can choose $y_0>0$ such that $\forall y \geq y_0 \ \abs{\Delta^{(2)} g^x_{y,1}} \leq \frac{C_x}{y^{4-x}}$. Moreover, for any fixed $\epsilon \in (0;1)$, one can choose $K_\epsilon$ such that
\begin{equation}\label{e:bound_on_C_x}
	\forall x \in \{4,5, \dots\} \ C_x \leq K_\epsilon (1+\epsilon)^x.
\end{equation}
\end{lem}

\begin{proof}
It is straightforward to check that $\Delta^{(2)} g^x_{y,1} = (y+4)^{x}-4(y+3)^x + 6(y+2)^x-4(y+1)^x+y^{x}$. Take $y\geq y_0>4$. Then Taylor's expansion in a neighborhood of zero yields\footnote{$\binom{x}{j} \bydef \frac{x(x-1)\cdots (x - j+1)}{j!},j \geq 1; \binom{x}{0} \bydef 1$} 
\begin{multline*}
	y^{-x}\Delta^{(2)} g^x_{y,1} = \sum_{j = 0}^\infty \binom{x}{j} y^{-j}\left(4^{j}-4 \cdot 3^{j} + 6 \cdot 2^{j} - 4 \right) +1 = \\
    \left[\text{terms corresponding to $j = 0,1,2,3$ cancel out}\right] =\\
    \sum_{j = 4}^\infty \binom{x}{j} y^{-j}\left(4^{j}-4 \cdot 3^{j} + 6 \cdot 2^{j} - 4 \right) = 
    \sum_{j = 4}^\infty \binom{x}{j} \left(\frac{4}{y}\right)^{j}a_j, 
\end{multline*}
where $a_j=\left(1-4 \cdot \left(\frac{3}{4}\right)^{j} + 6 \cdot \left(\frac{1}{2}\right)^{j} - \left(\frac{1}{4}\right)^{j-1} \right),j \geq 4$. Therefore,
\begin{equation*}
	\Abs{y^{-x}\Delta^{(2)} g^x_{y,1}}	 < \left(\frac{4}{y}\right)^{4}\sup_{j \geq 4} \abs{a_j} \sum_{j=4}^\infty\left \vert\binom{x}{j} \right \vert\left(\frac{4}{y_0}\right)^{j-4}.
\end{equation*}
The series on the right hand side converges and does not depend on $y$ whereas $\sup_{j \geq 4} \abs{a_j}$ is bounded. Suppose $x \in \{4,5,\dots \}$. Then
\begin{equation*}
	 \sum_{j=4}^\infty\binom{x}{j} \left(\frac{4}{y_0}\right)^{j-4} =
      \frac{y_0^4}{4^4}\sum_{j=4}^x\binom{x}{j}  \left(\frac{4}{y_0}\right)^{j}<\frac{y_0^4}{4^4}\left(1+\frac{4}{y_0}\right)^x,
\end{equation*}
and for a fixed $\epsilon \in (0;1)$ it suffices to choose $y_0$ such that $4y_0^{-1}< \epsilon$.
\end{proof}

\begin{proof}[Proof of proposition \ref{p:sfBm_and_bfBm}.] Since the proof of proposition is nothing more but a careful application of Taylor's formula, we give a detailed exposition for the sfBm. In case of other families it is a repetition of the latter with some necessary changes. For short we omit a subscript denoting that the quantities under consideration correspond to the sfBm, e.g. we write $R,l,\dots$ etc. instead of $R_{S^H},l_{S^H},\dots$. $H \in(0;1)$ is assumed to be fixed.

\emph{Step 1.} Let $t,s,h>0$ and $u \in [0;1)$. Then 
\begin{multline*}
	R(s,t)= s^{2H}+t^{2H}-\frac{1}{2}\left( \abs{s+t}^{2H}+\abs{t-s}^{2H}\right) \imply\\
    R(t,t+h)=t^{2H}+(t+h)^{2H} - \frac{1}{2}\left( (2t+h)^{2H} + h^{2H}\right) =\\
    t^{2H}\left( 1+ \left(1+\frac{h}{t}\right)^{2H} - \frac{1}{2}\left(\left(2+\frac{h}{t}\right)^{2H} + \left(\frac{h}{t}\right)^{2H} \right) \right) \imply \\
    l(u)=\frac{1}{2-2^{2H-1}}
    \left( 
    	1+ (1+u)^{2H} - \frac{1}{2}\left( (2+u)^{2H} + u^{2H}\right)
    \right) \imply \\
    l(u) - \frac{1}{2}\left(1+(1+u)^{2H}\right) = -\frac{1}{2}u^{2 \kappa}L^{2}(u) =\\
    \frac{2^{2H-1}}{2-2^{2H-1}}
    \left( 
    	\frac{1+(1+u)^{2H}}{2}-\left(1+\frac{u}{2} \right)^{2H}-\left(\frac{u}{2}\right)^{2H}
    \right) = \\
    \frac{2^{2H-1}}{2-2^{2H-1}}\left(
    	\sum_{k=1}^\infty\binom{2H}{k}u^{k}\left(\frac{1}{2}-\frac{1}{2^k} \right)- \left(\frac{u}{2}\right)^{2H}
    \right)=-\frac{u^{2H}}{4-4^{H}}\left(1+O\left(u^{2(1-H)}\right) \right).
\end{multline*}
with the last two equalities due to Taylor's expansion in the neighborhood of 0. Hence, under assumptions made, \emph{(L1)--(L2)} hold.

\emph{Step 2.} By \emph{Step 1},
\begin{equation*}
	p(u)= u^{2 \kappa}L^2(u)=\frac{2^{2H}}{2-2^{2H-1}}\left(
    \left(\frac{u}{2}\right)^{2H}+\left(1+\frac{u}{2}\right)^{2H}-\frac{1+(1+u)^{2H}}{2}
\right).    
\end{equation*}
Since constant does not affect the order of differences $\Delta p,\Delta^{(2)} p$, it suffices to show that condition \emph{(L3)} of the theorem \ref{t:clt} applies to $\tilde{p} \bydef \left(\frac{2^{2H}}{2-2^{2H-1}}\right)^{-1}p$. Let $g^{x}$ be the same as in lemma \ref{l:diff2_asymp}. Then
\begin{multline*}
	\Delta \tilde{p}_{ku,u} = \left(\frac{u}{2}\right)^{2H}\left((k+1)^{2H}-k^{2H}\right) +\\ \left(\left(1+(k+1)\frac{u}{2} \right)^{2H}-\left(1+k\frac{u}{2} \right)^{2H}\right)-\frac{1}{2}
    \left(\left(1+(k+1){u} \right)^{2H}-\left(1+k{u} \right)^{2H}\right) \imply \\
    \Delta^{(2)} \tilde{p}_{ku,u} = \left(\frac{u}{2}\right)^{2H}g^{2H}(k) + \tilde{g}_k\left(\frac{u}{2}\right)-\frac{1}{2}\tilde{g}_k\left({u}\right),
\end{multline*}
where $\tilde{g}_k\left({u}\right) = (1+(k+2)u)^{2H}-2(1+(k+1)u)^{2H}+(1+ku)^{2H}$. Next, note that
\begin{multline*}
	(1+ju)^{2H}\tilde{g}_{k-j}\left(\frac{u}{1+ju}\right) = (1+ju)^{2H}\Bigg[ \left(1+(k-j+2)\frac{u}{1+ju}\right)^{2H}-\\2\left(1+(k-j+1)\frac{u}{1+ju}\right)^{2H}+
    \left(1+(k-j)\frac{u}{1+ju}\right)^{2H}\Bigg] =\tilde{g}_k(u);
\end{multline*}
and
\begin{multline*}
    (1+ju)^{2H}\tilde{g}_{k-j}\left(\frac{u}{2(1+ju)}\right) = 
    (1+ju)^{2H}\Bigg[ \left(1+(k-j+2)\frac{u}{2(1+ju)}\right)^{2H}-\\
    2\left(1+(k-j+1)\frac{u}{2(1+ju)}\right)^{2H}+\left(1+(k-j)\frac{u}{2(1+ju)}\right)^{2H}\Bigg] = \tilde{g}_{k+j}\left(\frac{u}{2}\right).
\end{multline*}
Therefore
\begin{multline*}
	\Delta^{(2)} \tilde{p}_{ku,u} - 2(1+u)^{2H}\Delta^{(2)} \tilde{p}_{(k-1)\frac{u}{1+u},\frac{u}{1+u}}+(1+2u)^{2H}\Delta^{(2)} \tilde{p}_{(k-2)\frac{u}{1+2u},\frac{u}{1+2u}} =\\ \left(\frac{u}{2}\right)^{2H}\Delta^{(2)} g^{2H}_{k-2,1} + 
    \tilde{g}_k \left(\frac{u}{2}\right)-2\tilde{g}_{k+1} \left(\frac{u}{2}\right)+\tilde{g}_{k+2} \left(\frac{u}{2}\right).
\end{multline*}
By lemma \ref{l:diff2_asymp}, $\Delta^{(2)}  g^{2H}_{k-2,1} = O \left(\frac{1}{k^{4-2H}}\right)$. Next, let $v = \frac{u}{2}$. Fix $q \in(1/2;1), k_0 \in \Nd, \epsilon >0$ such that:
\begin{itemize}
\item $q(1+\epsilon)<1$;
\item $\forall j \in \{4,5, \dots \}, \forall k \in \{k_0, \dots,n-2\}\ \Delta^{(2)}g_{k,1}^{j} \leq \frac{K_\epsilon(1+\epsilon)^j}{k^{4-j}}$; 
\end{itemize}
with the last being possible due to lemma \ref{l:diff2_asymp}. Since $u \in [0; n^{-1}]$, $v(k+4) \leq \frac{n+4}{2n} \leq q$ provided $n$  is large enough. Consequently, $\forall k \in \{k_0, \dots, n-2\}$ and some $\theta^{\epsilon}_{j,k} \in [-1;1]$,
\begin{multline*}
	\tilde{g}_k \left(\frac{u}{2}\right)-2\tilde{g}_{k+1} \left(\frac{u}{2}\right)+\tilde{g}_{k+2} \left(\frac{u}{2}\right) = \\
    (1+(k+4)v)^{2H}-4(1+(k+3)v)^{2H}+6(1+(k+2)v)^{2H}-4(1+(k+1)v)^{2H}+(1+k)^{2H} =\\ \sum_{j=0}^{\infty}\binom{2H}{j}v^{j}\Delta^{(2)} g^{j}_{k,1} = [\text{since }\forall k \ \Delta^{(2)} g^{j}_{k,1} = 0, \text{ for }j=0,1,2,3] =\\
    K_\epsilon
    \frac{(v(1+\epsilon))^{2H}}{k^{4-2H}}\sum_{j=4}^{\infty}\binom{2H}{j}(vk(1+\epsilon))^{j-2H}\theta^\epsilon_{j,k}=\frac{u^{2H}}{k^{4-2H}}O(1).
\end{multline*}
Since $k \in \{2,\dots, k_0\}\imply\abs{\Delta^{(2)} g^j_{k,1}}\leq 16(k_0+4)^j$, using the same expressions as above,
\begin{equation*}
	\left\vert \tilde{g}_k \left(\frac{u}{2}\right)-2\tilde{g}_{k+1} \left(\frac{u}{2}\right)+\tilde{g}_{k+2} \left(\frac{u}{2}\right)\right\vert \leq 16 \sum_{j = 4}^{\infty} \binom{2H}{j}(v(k_0+4))^j = u^{4} O(1).
\end{equation*}
Hence, for all $k \in \{2,\dots,n\}$,
\begin{equation*}
	\left \vert\Delta^{(2)} \tilde{p}_{ku,u} - 2(1+u)^{2H}\Delta^{(2)} \tilde{p}_{(k-1)\frac{u}{1+u},\frac{u}{1+u}}+(1+2u)^{2H}\Delta^{(2)} \tilde{p}_{(k-2)\frac{u}{1+2u},\frac{u}{1+2u}}\right \vert \leq d\frac{u^{2H}}{k^{4-2H}},
\end{equation*}
with some constant $d$ independent of $k,n$.
\end{proof}

\bibliographystyle{alpha}

\end{document}